\documentclass[a4,11pt]{amsart}
\usepackage{amssymb,amstext,amsmath,amscd,amsthm,amsfonts,enumerate,graphicx,latexsym,color,stmaryrd,tikz}
\usepackage[all]{xy}

\newtheorem{thm}{Theorem}[section]
\newtheorem{cor}[thm]{Corollary}
\newtheorem{lem}[thm]{Lemma}
\newtheorem{prop}[thm]{Proposition}

\newtheorem*{thm*}{Theorem}

\theoremstyle{definition}
\newtheorem{defn}[thm]{Definition}
\newtheorem{rem}[thm]{Remark}

\newtheorem{exam}[thm]{Example}

\newtheorem*{claim*}{Claim}
\newtheorem*{ques*}{Question}

\newtheorem{chunk}[thm]{\hspace*{-1.065ex}\bf}
\newtheorem*{ac}{Acknowledgments}
\theoremstyle{remark}
\numberwithin{equation}{thm}
\def\fm{\mathfrak{m}}
\def\fn{\mathfrak{n}}

\def\fa{\mathfrak{a}}

\def\Im{\operatorname{Im}}

\def\Gdim{\operatorname{G-dim}}

\def\H{\operatorname{H}}
\def\Tor{\operatorname{Tor}}
\def\Ext{\operatorname{Ext}}

\def\grade{\operatorname{grade}}
\def\gr{\operatorname{gr}}
\def\Pn{\operatorname{P}}
\def\lin{\operatorname{lin}}
\def\ld{\operatorname{ld}}
\def\reg{\operatorname{reg}}

\def\E{\operatorname{E}}
\def\K{\operatorname{K}}

\def\Ext{\operatorname{Ext}}

\def\pd{\mathrm{pd}}

\begin{document}
\setlength{\baselineskip}{15pt}
\title{Some criteria for detecting large homomorphisms of local rings}
\author{Mohsen Gheibi}
\address{Mohsen Gheibi, Department of Mathematics, University of Texas at Arlington,
411 S. Nedderman Drive,
Pickard Hall 478,
Arlington, TX, 76019, USA }
\email{mohsen.gheibi@uta.edu}
\author{Ryo Takahashi}
\address{Ryo Takahashi, Graduate School of Mathematics, Nagoya University, Furocho, Chikusaku, Nagoya{, Aichi} 464-8602, Japan}
\email{takahashi@math.nagoya-u.ac.jp}
\urladdr{http://www.math.nagoya-u.ac.jp/~takahashi/}
\thanks{Ryo Takahashi was partly supported by JSPS Grant-in-Aid for Scientific
Research 16K05098 and JSPS Fund for the Promotion of Joint
International Research 16KK0099}

\begin{abstract}
We study ideals in a local ring $R$ whose quotient rings induce large homomorphisms of local rings. We characterize such ideals in terms of the maps of Koszul homologies over several classes of local rings, including complete intersections, Koszul rings, and some classes of Golod rings.
\end{abstract}

\maketitle
\section{Introduction}

In 1978 Avramov \cite{Av2} introduced and studied small homomorphisms of local rings, and a year after, Levin \cite{Lev} introduced large homomorphisms of local rings as a dual notion of the small homomorphisms. A surjective local homomorphism $f:R\to S$ is called \emph{small} if the induced homomorphism of graded algebras $f_*:\Tor^R_*(k,k) \to \Tor^S_*(k,k)$ is injective, and it is called \emph{large} if $f_*$ is surjective.  Levin proved that a local homomorphism $R \to S$ is large if and only if for every finitely generated $S$-module $M$ there is an equality of Poincar\'{e} series $\Pn^R_M(t)=\Pn^R_S(t)\cdot \Pn^S_M(t)$. This result makes  large homomorphisms  a very useful tool for understanding betti numbers and computing Poincar\'{e} series.

Small homomorphisms are closely related to another class of local homomorphisms namely \emph{Golod homomorphisms}; see Definition \ref{GH}. In fact, each Golod homomorphism is small, and these homomorphisms have been studied very well in several articles; see for example \cite{Av2}, \cite{Lev2}, \cite{Levin}, and \cite{S}. On the other hand, we are only aware of relatively few results about the large homomorphisms. The goal of this paper is not only to prove new results, but also to collect known facts and ideas about large homomorphisms that have been used in many articles without stating them. 
The organization of the paper is as follows.

 In section 2, after some preliminaries and examples, we give various conditions under which a surjective homomorphism $R \to S$ is large, specifically when $R$ is a complete intersection; see Theorem \ref{CI}. 

In section 3, {over a local ring $(R,\fm)$ with an ideal $I$, we give a necessary and sufficient condition when $R\to R/I$ is large and $R\to R/\fm I$ is small simultaneously; see Theorem \ref{Tor}.  As a result, we show that a sujective local homomorphism $R \to S$ where $S$ is a Koszul $R$-module is large; see Corollary \ref{exlin}. Then we {provide} a sufficient condition for large homomorphisms over Golod rings; see Proposition \ref{Massey}.


\section{Large homomorphisms over complete intersections}

Throughout $(R,\fm,k)$ is a commutative Noetherian local ring with maximal ideal $\fm$ and residue field $k$.
We start this section by recalling some definitions and preliminaries.

A \emph{differential graded algebra} (DG algebra) $A$ is a complex $(A,\partial)$ equipped with an algebra structure such that $ A_0$ is a commutative unitary ring, $A_i=0$ for $i<0$, and the product satisfies the Leibniz rule $\partial(ab)=\partial(a)b+(-1)^{|a|}a \partial(b)$, for $a,b \in A$.

\begin{chunk}{\bf Acyclic Closure.}\label{AC} Let $I$ be an ideal of $R$. An \emph{acyclic closure} $R\langle X_i |\ {i\geq 1} \rangle$ of the augmentation $R\to R/I$ obtained from the Tate construction is a DG algebra resolution for $R/I$ over $R$, where $X_i$ is a set of exterior variables when $i$ is odd, and it is a set of divided variables if $i$ is even number. Note that $X_i$ are added in a given homological degree $i$ in such a way that their images under the differential minimally generate the $(i-1)$th homology. For $i\geq 1$, we set $\pi_i(R)$ to be the $k$-vector space with basis $X_i$. By using the notations of \cite{Av}, we write $\varepsilon_i=\dim_k \pi_i(R)$. 

An acyclic closure $R\langle X_i \rangle$ of $R/I$ is not a minimal resolution in general  {but} if $I=\fm$, then it is  a minimal resolution by Gulliksen's theorem; see \cite[Theorem 1.6.2]{GL}.
\end{chunk}

Let $U$ be a DG algebra resolution of $k$ over $R$ described in \ref{AC}, and let $\fa$ be an ideal of $R$. Then $A=U\otimes_R R/\fa $ is a DG algebra, and hence, $\Tor^R(k,R/\fa)$ also has the graded algebra structure induced from $A$.

Let $M$ be a finitely generated $R$-module. The \emph{Poincar\'{e} series} $\Pn^R_M(t)$ of $M$ over $R$ is defined to be the formal power series $\Pn^R_M(t)=\sum^\infty_{i=0}\dim_k\Tor^R_i(M,k)t^i$.

\begin{defn}\cite[Theorem 1.1]{Lev}\label{D1}  Let $(R,\fm,k)$ be a local ring, and let $f:R\rightarrow S$ be a surjective local homomorphism. Then $f$ is called \emph{large} if one of the following equivalent conditions holds.
\begin{enumerate}
\item The induced map $f_*:\Tor^R_*(k,k)\rightarrow \Tor^S_*(k,k)$ of graded algebras is surjective.
\item The induced map $\varphi_*: \Tor^R_*(S,k)\rightarrow \Tor^R_*(k,k)$ of graded algebras is injective.
\item $\Pn^R_k(t)=\Pn^R_S(t)\cdot \Pn^S_k(t)$.
\end{enumerate}
\end{defn}

\begin{chunk}{\bf Necessary Condition.}\label{NC} Suppose $R\rightarrow R/I$ is a large homomorphism. Then by Definition \ref{D1}({2}) we have  {$\Tor^R_1(R/I,k) \to \Tor^R_1(k,k)$ is injective} which specifically tells us that {a minimal generating set of $I$ can be completed to a minimal generating set of $\fm$}, equivalently $I\cap \fm^2 = \fm I$. {We will need this condition in most of our results.}
\end{chunk}

In the following we list some well-known examples of large homomorphisms from the literature.

\begin{exam}\label{E1} Let $(R,\fm,k)$ be a local ring, and let $I$ be an ideal such that $I\cap \fm^2 = \fm I$. In either of the following, {the induced map} $R\rightarrow R/I$ is a large homomorphism.
\begin{enumerate}
\item $\pd_R (I)<\infty$.
\item $(0:_RI)=\fm$.
\item The ring $R/I$ is complete intersection.
\item The map ${R \to R/}I$ is a quasi-complete intersection {homomorphism}.
\item The ring $R$ is Cohen-Macaulay non-Gorenstein, $\fm^2\subseteq I$, and $\Gdim_R(I)<\infty$.
\item $\fm=I\oplus J$.
\end{enumerate}
\end{exam}
\begin{proof} (1) We show that if $\pd_R(I)<\infty$ then $I$ is generated by a regular sequence. 
The claim is obvious if $I$ is a principal ideal. Hence we may assume $I=(x_1,\dots,x_n)$ such that $n\geq 2$ and $x_i\in \fm \setminus \fm^2$ for all $i$. Since $\pd_R(I)<\infty$, we have $\grade_R(I)>0$ by \cite[Corollary 1.4.6]{BH}. By prime avoidance, one can choose a non zero-divisor $x\in I \setminus \fm^2$. It follows from \cite[Theorem 2.2.3]{Av} that $\pd_{R/(x)}R/I<\infty$. Since $I/(x)$ is generated by $n-1$ elements, by induction it is generated by a regular sequence over $R/(x)$. Hence $I$ is generated by a regular sequence over $R$. Therefore $R\rightarrow R/I$ is a large homomorphism by \cite[Theorem 2.2]{Lev}. 

(2) also follows by \cite[Theorem 2.2]{Lev}.
See \cite[Theorem 2.4]{Lev} for (3), \cite[Theorem 6.2]{AHS} for (4), and \cite[Theorem 1.2]{GT} for (5).

To see (6), consider the exact sequence $$0\to R \to R/I\oplus R/J \to k \to 0.$$ Applying $-\otimes_Rk$, we see that the induced map $\Tor^R_i(R/I,k)\to \Tor^R_i(k,k)$ is injective for all $i\geq 0$.
\end{proof}

The following provides more examples of large homomorphisms.

\begin{rem}\label{E3} Let $R$ be a local ring and let $I$ be an ideal of $R$. If there exists a local ring $Q$ and a surjective local homomorphism $Q \to R$ such that the composition map $Q \to R/I$ is a large homomorphism then $R \to R/I$ is large. Indeed, since the composition of the maps $Q \to R \to R/I$ is large, the induced composition of homomorphisms $\Tor^{Q}_*(k,k) \to \Tor^{R}_*(k,k) \to \Tor^{R/I}_*(k,k)$ is surjective. Hence $\Tor^{R}_*(k,k) \to \Tor^{R/I}_*(k,k)$ is surjective.
\end{rem}

By using Example \ref{E1}(6) and Remark \ref{E3} we get the following.

\begin{exam} Let $R=S \#_k T$  be the connected sum of local rings $S$ and $T$; see \cite{AAM}. Then $R \to S$ is a large homomorphism.
To see that, let $Q=S\times_kT$ be the fiber product of $S$ and $T$. Then the composition of homomorphisms $Q\to R\to S$ is large by Example \ref{E1} (6).
Hence $R\to S$ is large by Remark \ref{E3}.
\end{exam}

\begin{chunk}{\bf Change of ring spectral sequence.}\label{DGS} Let $f:R\to S$ be a local homomorphism. Consider the change of ring spectral sequence $$\E^2_{p,q}\cong\Tor^S_p(k,k)\otimes \Tor^R_q(S,k)\Longrightarrow \Tor^R_{p+q}(k,k).$$
The spectral sequence is derived from the double complex $C_{p,q}=G_p\otimes_RF_q$ where $F$ and $G$ are free resolutions of $k$ over $R$ and $S$, respectively. 
The edge homomorphisms $\Tor^R_p(k,k)\rightarrow \Tor^S_p(k,k)\cong \E^2_{p,0}$ and $\E^2_{0,q}\cong \Tor^R_q(S,k)\rightarrow \Tor^R_q(k,k)$ are the homomorphisms induced by $f: R\to S$ and $\varphi: S\to k$, respectively; see \cite[page 348]{CE}.
\end{chunk}

\begin{chunk}{\bf Koszul homology} Let $(R,\fm,k)$ be a local ring and let $I$ be an ideal of $R$. We denote $\K(I)$ the Koszul complex with respect to a minimal generating set of $I$, and $\H_i(I):=\H_i(\K(I))$. If $I=\fm$ we write $\K(R):=\K(\fm)$ and $\H_i(R):=\H_i(\fm)$.  {If $I\cap \fm^2=\fm I$ then a minimal generating set of $I$ can be completed to one for $\fm$. In this case,} $\K(I)$ is a subcomplex of $\K(R)$. Therefore the inclusion map induces a natural homomorphism $\H_*(I) \to \H_*(R)$ of homology algebras.
\end{chunk}

Large homomorphisms are well understood over complete intersection local rings due to Gulliksen and Levin \cite{GL}. It is worth to bring here those conditions which characterize large homomorphisms over complete intersection local rings. First we need the following lemmas.

\begin{lem}\label{KH} Let $(R,\fm,k)$ be a local ring and let $I$ be an ideal of $R$ such that $I\cap \fm^2=\fm I$. Then the induced map \emph{$\H_1(R)\rightarrow \H_1(R/I)$} is surjective.
\end{lem}
\begin{proof} Set $S=R/I$. The assumption $I\cap \fm^2=\fm I$ implies that $\Tor^R_1(S,k) \to \Tor^R_1(k,k)$ is injective. Hence by \ref{DGS}, $d^2_{2,0}=0$ and therefore the induced map $\Tor^R_i(k,k)\rightarrow \Tor^S_i(k,k)$ is surjective for $i\leq 2$. Let $\K(R)$ and $\K(S)$ respectively be the Koszul complexes of $R$ and $S$. Let $\sigma_1,\dots, \sigma_r\in \K(R)$ and $\delta_1,\dots,\delta_s\in \K_1(S)$ be cycles whose classes are basis for the vector spaces $\H_1(R)$ and $\H_1(S)$, respectively. Then by the Tate construction, there exists a commutative diagram
$$
\begin{CD}
(\oplus_{i<j\leq n}Re_i\wedge e_j)\oplus(\oplus_{\ell=1}^rRT_\ell) @>>> \oplus_{i=1}^n Re_i @>>> R @>>> k @>>> 0 \\
@VVf_2V @VVf_1V @VVf_0V @VV=V \\
(\oplus_{i<j\leq m}Se_i\wedge e_j)\oplus(\oplus_{\ell=1}^sS\tau_\ell) @>>> \oplus_{i=1}^m Se_i @>>> S @>>> k @>>> 0,
\end{CD}
$$
where the rows are beginning of the free resolutions of $k$ over $R$ and $S$, $T_1,\dots,T_r$ and $\tau_1,\dots \tau_s$ are divided variables of homological degree $2$ with $\partial T_i=\sigma_i$ and $\partial \tau_i=\delta_i$; see \ref{AC}. Since $\Tor^R_2(k,k)\rightarrow \Tor^S_2(k,k)$ is surjective, $f_2$ is surjective too, and hence, $\tau_j=f_2(\eta_j)$ for some $\eta_j \in F_2$. Thus $\delta_j=f_1(\partial\eta_j)$. Since the map $\H_1(R)\rightarrow \H_1(S)$ is induced by $f_1$, we are done.
\end{proof}

\begin{lem}\label{inj} Let $I$ be an ideal of $R$ such that $I\cap \fm^2=\fm I$. If the natural map $\H_1(I)\otimes_R k \to \H_1(R)$ is injective, then the induced map $\Tor^R_2(R/I,k) \to \Tor^R_2(k,k)$ is injective.
\end{lem}
\begin{proof} Let $\fm=(x_1,\dots,x_n)$ and $I=(x_1,\dots,x_l)$ with $l \leq n$. Let $z_1,\dots, z_q\in \K_1(I)$ and $z'_1,\dots,z'_r\in \K_1(R)$ be cycles whose classes minimally generate $\H_1({I})$ and $\H_1(R)$, respectively. Then there exists a commutative diagram
$$
\begin{CD}
F_2=(\oplus_{i<j\leq l}Re_i\wedge e_j)\oplus(\oplus_{\ell=1}^qRS_\ell) @>>> F_1=\oplus_{i=1}^l Re_i @>>> R @>>> R/I @>>> 0 \\
@VV\phi_2V @VV\phi_1V @VV=V @VV\pi V \\
G_2=(\oplus_{i<j\leq n}Re_i\wedge e_j)\oplus(\oplus_{\ell=1}^rRT_\ell) @>>> G_1=\oplus_{i=1}^n Re_i @>>> R @>>> k @>>> 0,
\end{CD}
$$ where the rows are beginning of Tate resolutions of $R/I$ and $k$, $\phi_1$ is natural inclusion, $S_1,\dots S_l$ and $T_1,\dots,T_r$ are divided variables of homological degree $2$ with  $\partial S_i=z_i$ and $\partial T_i=z'_i$. 
Note that $\H_1(I) \to \H_1(R)$ is induced by $\phi_1$ and one has $\phi_2|_{\oplus_{i<j\leq n}Re_i\wedge e_j}$ is the natural injection $\K(I)\to \K(R)$. Assume $\H_1(I)\otimes_R k \to \H_1(R)$ is injective and let $\bar{\phi_1}$ be the induced map $\H_1(I)\to \H_1(R)$. 
Then for any choice of $\lambda_1,\dots,\lambda_q$ such that $\lambda_j \in R\setminus \fm$ for some $j$, one has $\bar{\phi_1}(\lambda_1\bar{z_1}+\cdots+\lambda_q\bar{z_q})$ is non zero in $\H_1(R)$ where $\bar{z_i}$ is the class of $z_i$ in $\H_1(I)$. This means $\phi_1(\lambda_1z_1+\cdots+\lambda_qz_q)\notin \Im \partial^{K(R)}_2$ and therefore one has $\phi_2(\lambda_1S_1+\cdots+\lambda_qS_q)\notin (\oplus_{i<j\leq n}Re_i\wedge e_j)\oplus(\oplus_{\ell=1}^r\fm T_\ell)$.
This implies that $\phi_2$ is split injective and hence the induced map $\Tor^R_2(R/I,k) \to \Tor^R_2(k,k)$ is injective.
\end{proof}

\begin{thm}\label{CI} Let $(R,\fm,k)$ be a complete intersection local ring, and let $I$ be an ideal of $R$ with $I\cap \fm^2=\fm I$.
Let $S=R/I$. Then the following are equivalent.
\begin{enumerate}
\item {The induced map} $R\rightarrow S$ is large.
\item {The ring} $S$ is complete intersection.
\item The induced map {$\Tor^R_2(S,k)\rightarrow \Tor^R_2(k,k)$} is injective.
\item The induced map {$\Tor^R_3(k,k)\rightarrow \Tor^S_3(k,k)$} is surjective.
\item The induced map $\H_1(I)\otimes_Rk \to \H_1(R)$ is injective.
\item The induced map {$\H_2(R)\rightarrow \H_2(S)$} is surjective.
\end{enumerate}
\end{thm} 
\begin{proof} (1)$\Rightarrow$(2) Since $R$ is complete intersection, we have $\varepsilon_3(R)=0$ by \cite[Theorem 7.3.3]{Av}. Since $R\rightarrow S$ is large, we get $\varepsilon_3(S)=0$. Therefore $S$ is complete intersection by \cite[Theorem 7.3.3]{Av}.

(2)$\Rightarrow$(1) This follows from Example \ref{E1}(3).

(2)$\Rightarrow$(5) Assume $S$ is complete intersection. Then the map $R\rightarrow S$ is a quasi-complete intersection homomorphism by \cite[Proposition 7.7]{AHS}. By \cite[Theorem 5.3]{AHS} there exists an exact sequence 
$0 \to \H_1(I)\otimes_Rk \to \pi_2(R)$ of $k$-vector spaces. Since $\pi_2(R)\cong \H_1(R)$, we have the induced map $\H_1(I)\otimes_Rk \to \H_1(R)$ is injective; see \ref{AC}.

(5)$\Rightarrow$(3) This follows from Lemma \ref{inj}.

(3)$\Leftrightarrow$(4) This follows from the change of rings spectral sequence; see \ref{DGS}.

(4)$\Rightarrow$(2) Same argument as (1)$\Rightarrow$(2) applies here as well.

(2)$\Leftrightarrow$(6) {Since $I\cap \fm^2=\fm I$, we have $\H_1(R)\to \H_1(S)$ is surjective by Lemma \ref{KH}.} Since $R$ is complete intersection, by the Tate-Assmus Theorem \cite[Theorem 2.3.11]{BH} $\H_2(R)=\H_1(R)^2$.  {Therefore} $S$ is complete intersection if and only if $\H_2(S)=\H_1(S)^2$ if and only if $\H_2(R)\rightarrow \H_2(S)$ is surjective.
\end{proof}

\begin{rem}\label{P2} Let $f:R\rightarrow S$ be a surjective homomorphism of local rings. Then $f$ is large if $f_i:\Tor^R_i(k,k) \rightarrow \Tor^S_i(k,k)$ is surjective for all $i\gg 0$. Indeed, this is easy to see when $S$ is a regular ring.  In this case, for example, the sujectivity of $\K(R)\to \K(S)$  implies that $f$ is large. Suppose $S$ is singular.
One has $f_i$ is surjective if and only if $f^i:\Ext^i_S(k,k) \rightarrow \Ext^i_R(k,k)$ is injective. It is well-known that $\Ext^*_S(k,k)$ is the universal enveloping algebra of the homotopy Lie algebra $\pi^*(S)$; see \cite[Theorem 10.2.1]{Av}. Then by \cite[Lemma 5.1.7]{AV}, any element $\chi \in \pi^2(S)$ (and hence any powers of $\chi$) is a non zero divisor on $\Ext^*_S(k,k)$. Let $\alpha \in \ker{f^j}$ for some $j$. One has $\chi^i \alpha \in \ker f^{2i+j}$. Since $f^i$ is injective for all $i\gg 0$, we have $\chi^i \alpha=0$ and hence $\alpha =0$.
\end{rem}

Regarding {Remark} \ref{P2}, one may ask whether $R\to S$ is large if the induced map $\Tor^R_i(S,k) \to \Tor^R_i(k,k)$ is injective for all $i \gg 0$. This is not true in general. For example, if $I\subseteq \fm^2$ is any non zero ideal of finite projective dimension then $R\to R/I$ is not large, while $\Tor^R_i(R/I,k) \to \Tor^R_i(k,k)$ is injective for all $i \gg 0$. We don't know whether $R\to R/I$ is large if $\pd_R(I)=\infty$ and $\Tor^R_i(R/I,k) \to \Tor^R_i(k,k)$ is injective for all $i \gg 0$. However, we are able to prove the following.

\begin{prop}\label{nonzero} Let $(R,\fm,k)$ be a local ring and let $I$ be an ideal of $R$ with $I\cap \fm^2=\fm I$. Assume the induced map $\H_1(I)\to \H_1(R)$ is non zero. If $\varphi_i:\Tor^R_i(R/I,k) \to \Tor^R_i(k,k)$ is injective for all $i \gg 0$ then $R\to R/I$ is large.
\end{prop}
\begin{proof} Let $\zeta$ be an element in $\H_1(I)$ whose image under the natural map $\H_1(I) \to \H_1(R)$ is non zero, and let $z\in \K_1(I)$ be a cycle whose class in $\H_1(I)$ is $\zeta$. By using the natural injection $\K_1(I)\hookrightarrow \K_1(R)$ we consider $z$ as an element of $\K_1(R)$. Then one checks that $z=z'+\partial(a)$ for some $z'\in I\K_1(R)$ and $a\in \K_2(R)$. Therefore $[z]=[z']$ in $\H_1(R)$. Let $U=R\langle X_i\rangle_{i\geq 1}$ be an acyclic closure of $k$ over $R$, and let ${X}\in U$ be a divided variable of degree 2 such that $\partial({X})=z$; see \ref{AC}. By \cite[Lemma 6.3.3]{Av}, there exists a {chain} $\Gamma$-derivation $d:U\to U$ of degree $-2$ which is trivial on $\K(R)$ and $d({X}^{(i)})={X}^{(i-1)}$.

Set $S=R/I$ and $A=U\otimes_RS$. Then ${X}\otimes 1_S$ is a cycle in $A$ whose class $[{X}\otimes 1_S]$ is non zero in $\Tor^R_2(S,k)$ (because $\varphi_2([{X}\otimes 1_S])={X}\otimes 1_k\neq 0$). Let $\alpha\otimes 1_S \in A$ be such that $[\alpha \otimes 1_S]\in \Tor^R_j(R/I,k)$ and $\varphi_j([\alpha \otimes 1_S])=0$. Since $\varphi_*$ is homomorphism of graded algebras, we have $\varphi_{j+2i}([{X}^{(i)}\alpha \otimes 1_S])=0$. Thus by assumption, $[{X}^{(i)}\alpha \otimes 1_S]=0$ for $i\gg0$. If $[d(\alpha)\otimes 1_S]=0$ then we have $[\alpha\otimes 1_S] = [d^i({X}^{(i)}\alpha) \otimes 1_S] =0$ and we are done.

Suppose $[d(\alpha)\otimes 1_S]\neq 0$, and let $r$ be the biggest integer such that $[d^r(\alpha)\otimes 1_S]\neq 0$. Since $\varphi_j([\alpha \otimes 1_S])=0$ and $d$ commutes with differentials, one has $\varphi_{j-2r}([d^r(\alpha)\otimes 1_S])=0$. By replacing $[\alpha \otimes 1_S]$ with $[d^r(\alpha)\otimes 1_S]$, the same argument as above shows that $[d^r(\alpha)\otimes 1_S]=0$ which is a contradiction.
\end{proof}

\section{Large homomorphisms over Koszul rings and Golod rings}

Let $(R,\fm,k)$ be a local ring and let $M$ be a finitely generated $R$-module. Let $F$ be a minimal free resolution of $M$ over $R$ and let $\lin^R(F)$ be the associated graded complex of $F$; see \cite[\S 1]{HI} and \cite[\S 2]{S} for more details. The \emph{linearity defect} $\ld_R(M)$ of $M$ over $R$ is defined by $$\ld_R(M)=\sup\{i|\ \H_i(\lin^R(F))\neq 0 \}.$$

\begin{chunk}{\bf Koszul Modules.} An $R$-module $M$ is called a \emph{Koszul module} (or is said to have a \emph{linear resolution}) if $\lin^R(F)$ is acyclic, equivalently $\ld_R(M)=0$. $R$ is called a \emph{Koszul ring} if $k$ is a Koszul $R$-module.

Let $\gr_{\fm}(R)=\oplus_{i\geq 0}\fm^i/\fm^{i+1}$ be the associated graded ring of $R$ and $\gr_\fm(M)=\oplus_{i\geq 0}\fm^i M/\fm^{i+1}M$ be the associated graded module of $M$. If $M$ is a Koszul module then $\lin^R(F)$ is a minimal free resolution of $\gr_\fm(M)$ over $\gr_{\fm}(R)$;
see \cite[Proposition 1.5]{HI}.

The \emph{regularity} $\reg_R(M)$ is defined to be the regularity of the graded module $\gr_\fm(M)$ over the graded ring $\gr_\fm(R)$; see \cite{S}. It follows that $M$ is Koszul if and only if $\reg_R(M)=0$.
\end{chunk}

\begin{defn} Let $k$ be a field, and let $A$ be a DG algebra with $\H_0(A)\cong k$. The algebra $A$ is said to admit a \emph{trivial Massey operation},
if for some $k$-basis ${\bf h}=\{h_\lambda\}_{\lambda \in \Lambda}$ of $\H_{\geq 1}(A)$ there exists $\mu : \coprod_{i\geq 1}{\bf h}^i \to A$ such that
\begin{enumerate}
\item $\mu(h_\lambda)=z_\lambda$ where $z_\lambda$ is a cycle in $A$ with class $[z_\lambda]=h_\lambda \in \H_{\geq 1}(A)$, and
\item $\partial \mu(h_{\lambda_1},\dots,h_{\lambda_n})= \displaystyle\sum^{n-1}_{i=1}\overline{\mu(h_{\lambda_1},\dots,h_{\lambda_i})}\mu(h_{\lambda_{i+1}},\dots,h_{\lambda_n})$, where $\overline{a}=(-1)^{|a|+1}a$.
\end{enumerate}
\end{defn}

\begin{defn}\label{GH} A surjective local homomorphism $f: R \to S$ is called \emph{Golod} if $$\Pn^S_k(t)=\frac{\Pn^R_k(t)}{1-t(\Pn^R_S(t)-1)},$$ or equivalently, the DG algebra $A = U\otimes_R S$ admits a trivial Massey operation, where $U$ is a minimal DG algebra resolution of $k$ over $R$; see \cite[Theorem 1.5]{Lev2}.
\end{defn}

{\begin{thm}\label{Tor} Let $(R,\fm,k)$ be a local ring, and let $I$ be an ideal of $R$ such that $I\cap \fm^2=\fm I$. Then the following conditions are equivalent.
\begin{enumerate}
\item The map \emph{$\Tor^R_i(\fm I,k)\rightarrow \Tor^R_i(I,k)$} induced by the inclusion $\fm I\hookrightarrow I$ is zero for all $i\geq 0$.
\item The homomorphism $\phi: R\rightarrow R/I$ is large, and the homomorphism $\psi: R\rightarrow R/\fm I$ is small.
\end{enumerate}
{Moreover, under these equivalent conditions $\psi$ is a Golod homomorphism.}
\end{thm} }
\begin{proof}
The assumption $I\cap \fm^2=\fm I$ implies an exact sequence $0\rightarrow I/\fm I \overset{\iota}\rightarrow \fm/\fm^2 \rightarrow \fm/(\fm^2+I) \rightarrow 0$ of $k$-vector spaces. By applying $-\otimes_Rk$ to the commutative diagram
$$
\begin{CD}
0 @>>> \fm I @>>> I @>>> I/\fm I @>>> 0 \\
@. @V VV @VVV @V\iota VV \\
0 @>>> \fm^2 @>>> \fm @>>> \fm/\fm^2 @>>> 0,
\end{CD}
$$
we get a commutative diagram
$$
\begin{CD}
 \Tor^R_i(\fm I,k) @>f_i>> \Tor^R_i(I,k) @>\gamma_i>> \Tor^R_i(I/\fm I,k) \\
@VVV @Vg_iVV @V\iota_iVV \\
\Tor^R_i(\fm^2,k) @>>> \Tor^R_i(\fm,k) @>>> \Tor^R_i(\fm/\fm^2,k)
\end{CD}
$$ with exact rows, where $\iota_i$ is injective.

(1)$\Rightarrow$ (2): Since $f_i$ is zero map and $\iota_i$ is split injection, the composition $\iota_i \circ \gamma_i$ is injective.
Therefore, the map $g_i$ is injective for all $i\geq 0$. Therefore $\Tor^R_i(R/I,k)\rightarrow \Tor^R_i(k,k)$ is injective for all $i \geq 0$.

Next, we show that $\psi$ is {Golod and hence it is} small{; see \cite[Definition 1.1]{Levin}}. {The argument is similar to the proofs of} \cite[Lemma 1.2]{RS} and \cite[Lemma 2.1]{Ah} {but we bring it for the reader's convenience.}  Let $U\rightarrow k$ be a minimal DG algebra resolution of $k$ over $R$, and set $A=U\otimes_RR/\fm I$. Then $\H_i(A)\cong \Tor^R_i(R/\fm I, k)$. The map $\Tor^R_i(R/\fm I,k)\rightarrow \Tor^R_i(R/I,k)$ is identified with the natural map $\H_i(A)\rightarrow \H_i(A/IA)$ for all $i$. Since $f_i=0$, every element in $\H_{i>0}(A)$ can be represented by $[x]$ for some $x\in IA$. Thus for any choice of $[z], [w]\in\H_{>0}(A)$ one has $z\cdot w \in I^2 A = 0$. Hence $A$ admits a trivial Massey operation; see \cite[Lemma 1.2]{Lev2}.

(2) $\Rightarrow$ (1): Suppose the map $\phi:R\rightarrow R/I$ is large, and the map $\psi: R\rightarrow R/\fm I$ is small homomorphisms.  { There exist a commutative diagram 
$$
\begin{CD}
 \Tor^R_i(R/\fm I,k) @>h_i>> \Tor^R_i(k,k) @>>> \Tor^R_{i-1}(\fm/\fm I,k) \\
&&  @V\psi_iVV @VVV \\
&& \Tor^{R/\fm I}_i(k,k) @>\cong>> \Tor^{R/\fm I}_{i-1}(\fm/\fm I,k),
\end{CD}
$$
for all $i\geq 1$. Since $\psi_i$ is injective for all $i\geq 0$, the commutative diagram implies that $h_i=0$ for all $i\geq 1$.}
Therefore, in the commutative diagram above, we have $g_i\circ f_i=0$ for all $i \geq 0$. Since $\phi$ is large, $g_i$ is injective and therefore $f_i=0$ as desired.
\end{proof}

The following is an immediate consequence of Theorem \ref{Tor}.

\begin{cor}\label{exlin}  Let $(R,\fm,k)$ be a local ring, and let $I$ be an ideal of $R$ with $I\cap \fm^2 =\fm I$. If $R/I$ is a Koszul $R$-module then $R\rightarrow R/I$ is a large homomorphism {and $R \to R/\fm I$ is a Golod homomorphism}. 
\end{cor}
\begin{proof} Since $I\cap \fm^2 =\fm I$ and $R/I$ is a Koszul $R$-module, $\reg_R(R/I)=0$. Therefore the maps $\Tor^R_i(\fm I,k)\rightarrow \Tor^R_i(I,k)$ are zero for all $i\geq 0$ by \cite[Theorem 3.2]{S}. Now the result follows from Theorem \ref{Tor}.
\end{proof}

\begin{cor}\label{GK} Let $(R,\fm,k)$ be a graded Koszul local ring and let $I$ be a homogeneous ideal of $R$. If the map ${R\to} R/I$ is large, then $R\to R/\fm I$ is a Golod homomorohism and $R/\fm I$ is a Koszul ring.
\end{cor}
\begin{proof} Set $S=R/\fm I$.  Since $R$ is Koszul, we have $\reg_R(k)=0$, and since $R \to R/I$ is large, we have the induced map $\Tor^R_*(R/I,k) \to \Tor^R_*(k,k)$ is injective. This implies that $\reg_R(R/I)=0$  and hence it is a Koszul $R$-module.
Therefore the map $R \to S$ is a Golod homomorphism by Corollary \ref{exlin}.

Next, we show that $S$ is Koszul. Since $R\to R/I$ is large, we have $I\cap \fm^2=\fm I$ by \ref{NC}. Then we may assume  
$\fm=(x_1,\dots,x_r,x_{r+1},\dots,x_n)$ such that $I=(x_1,\dots, x_r)$. Let $J=(x_{r+1},\dots,x_n)$, and let $\bar{I}$ and $\bar{J}$ be ideals of $S$ generated by images of $I$ and $J$, respectively. Then one easily checks that $\bar{I}\cap\bar{J}=0$ and therefore $S$ is the fiber product of $S/\bar{I}$ and $S/\bar{J}$. We have $S/\bar{I} \cong R/I$. Since $R$ is Koszul and $R\to R/I$ is large, the surjectivity of $\Tor^R_*(k,k)\to \Tor^{R/I}_*(k,k)$ shows that $R/I$ is Koszul.
On the other hand, $S/\bar{J}$ is isomorphic to a graded local ring $(T,\fn)$ with $\fn^2=0$ which is Koszul.
Now, by \cite[Proposition 3.11]{M}, $S$ is Koszul.
\end{proof} 

A Koszul local ring $R$ is called \emph{absolutely Koszul} if $\ld_R(M)<\infty$ for every finitely generated $R$-module $M$.

\begin{cor}\label{AK} Let $(R,\fm,k)$ be a graded Koszul complete intersection local ring and let $I$ be an ideal of $R$. If ${R\to} R/I$ is {large} then $R/\fm I$ is absolutely Koszul.
\end{cor}
\begin{proof} This follows by Corollary \ref{GK} and \cite[Theorem 5.9]{HI}.
\end{proof}

\begin{rem}\label{short} Let $(R,\fm,k)$ be a local Gorenstein ring with $\fm^3=0$. Let $I$ be a non zero ideal of $R$ such that $I\nsubseteq \fm^2$. {Then by \cite[Corollary 4.7]{AIS} $R/I$ is a Koszul $R$-module. Therefore the map $R\rightarrow R/I$ is large and $R\to R/(0:_R\fm)$ is Golod by Corollary \ref{exlin}. The later has been proven for all Artinian Gorenstein rings; see \cite[Theorem 2]{AL}.
Note that $R \to R/I$ may not be large if $R$ is not Gorenstein; see \cite[Example 3.12]{GT}.}
\end{rem}

\begin{exam}\label{E2} Let $R=k[x,y,z]/(x^2,y^2,z^2)$, where $k$ is a field.
Then $R$ is an Artinian Koszul complete intersection with $\fm^4=0$. Consider the ideal $I=(x+y+z)$ of $R$. Then $R/I\cong k[y,z]/(y^2,z^2,yz)$ which is not a complete intersection. Therefore $R\rightarrow R/I$ is not large by Proposition \ref{CI}.
If $k$ has characteristic different than $2$ then $\fm I=(xy,xz,yz)$ and therefore $R/\fm I=R/\fm^2$. This shows that the converse of Corollary \ref{AK} is not true. Also, one checks that $R\to R/\fm I$ is a Golod homomorphism. Therefore in Theorem \ref{Tor}, Golodness of $R\to R/\fm I$ does not guarantee that $R\to R/I$ is large.
\end{exam}

\begin{prop} Let $(R,\fm,k)$ be a local ring and let $I$ be an ideal of $R$. Assume $\ld_R(R/I)=l<\infty$. If the induced maps $\H_1(I) \to \H_1(R)$ and $\Tor^R_l(R/I,k) \to \Tor^R_l(k,k)$ are respectively non zero and injective, then the map $R\to R/I$ is large.
\end{prop}
\begin{proof} Let $F$ and $G$ respectively be minimal free resolutions of $R/I$ and $k$ over $R$.
Let $f: F \to G$ be the comparison homomorphism induced by $R/I \to k$. We show by induction that $f_i$ is split injective for all $i\geq l$. The injectivity of $\Tor^R_l(R/I,k) \to \Tor^R_l(k,k)$ settles the case $i=l$.
Let $n> l$ and suppose $f_i$ is split injective for $l\leq i < n$. Consider the commutative diagram 
$$
\begin{CD}
\dots @>>> F_n @>\partial^F_n>> F_{n-1} @>>> \dots\\
@. @V f_n VV @V f_{n-1} VV \\
\dots @>>> G_n @>\partial^G_n>> G_{n-1} @>>> \dots
\end{CD}
$$
and let $e\in F_n \setminus \fm F_n$. If $f_n(e) \in \fm G_n$ then $f_{n-1} \partial^F_n(e) \in \fm^2 G_{n-1}$. Since $f_{n-1}$ is split injective, we have $\partial^F_n(e) \in \fm^2 F_{n-1}$. As $\ld_R(R/I)<n$, this is a contradiction and hence $f_n(e) \in G_n\backslash \fm G_n$. 
Therefore the induced map $\Tor^R_i(R/I,k) \to \Tor^R_i(k,k)$ is injective for all $i\geq l$. Now by using Proposition \ref{nonzero}, the induced map $R\to R/I$ is large.
\end{proof}

\begin{defn} A local ring $(R,\fm,k)$ is called \emph{Golod} if the Poincar\'{e} series of $k$ over $R$ has presentation
$$\Pn^R_k(t)=\frac{(1+t)^{\nu_R(\fm)}}{1-\sum_{i\geq 1}\dim_k(\H_i(R))t^{i+1}},$$
where $\nu_R(\fm)$ is the number of minimal generating set of $\fm$.
\end{defn}

The following gives a sufficient condition for large homomorphisms over Golod rings by using the Koszul homologies.

\begin{prop}\label{Massey} Let $(R,\fm, k)$ be a Golod local ring, and let $f:R\to S$ be a surjective local homomorphism. 
If the natural map \emph{$\H_i(R)\rightarrow \H_i(S)$} is surjective for all $i\geq 1$, then $f$ is a large homomorphism.
\end{prop}
\begin{proof} Note that $R$ is Golod if and only if $\K(R)$ admits a trivial Massey operation; see \cite[Theorem 5.2.2]{Av} and \cite[Corollary 4.2.4]{GL}. Since $R$ is Golod, $\K(R)$ admits a trivial Massey operation. The surjectivity of $\H_i(R)\rightarrow \H_i(S)$ implies that $\K(S)$ admits a trivial Massey operation too and therefore it is a Golod ring. Let $F$ and $G$ respectively be minimal free resolutions of $k$ over $R$ and $S$ described in \cite[Theorem 5.2.2]{Av}. Then there exists a natural chain map $f:F\rightarrow G$ which only depends on the choice of $k$-bases of $\H_{i\geq 1}(R)$ and $\H_{i\geq 1}(S)$. Since $\H_i(R)\rightarrow \H_i(S)$ is surjective, $f$ is surjective as well. Therefore the induced map $\Tor_{*}^R(k,k)\rightarrow \Tor^S_*(k,k)$ is surjective as well. This finishes the proof.
\end{proof}

Recently, Gupta proved that if $R$ is a Golod local ring and a homomorphism $f: R\to S$ is large, then $S$ is Golod; see \cite[Theorem 1.5]{G}. This result provides a useful tool to detect Golod rings by using large homomorphism. 

\begin{exam} Let $R=\displaystyle\frac{k[x,y,z]}{(x^2,xy,xz,y^2,z^2)}$. Then $R$ is an Artinian  local ring with the maximal ideal $\fm=(x,y,z)$.
We have $R/(x)\cong k[y,z]/(y^2,z^2)$ is a complete intersection of codimension two. Therefore $R\rightarrow R/(x)$ is large by Example \ref{E1}(3), and $R/(x)$ is not Golod. Hence $R$ is not Golod.
\end{exam}

\begin{lem}\label{power} Let $(Q,\fn,k)$ be a regular local ring, $R=Q/\fn^p$ with $p\geq 2$, and let $\fm$ be the maximal ideal of $R$. Let $I$ be an ideal of $R$ such that $I\cap \fm^2=\fm I$. Then the natural map \emph{$\fm^{p-1}\K_i(I)\rightarrow \H_i(I)$} is surjective and therefore splits for all $i \geq 1$.
\end{lem}
\begin{proof} Let $x_1,\dots,x_n$ be a minimal generating set of $\fn$, and let $\overline{x}_i$ be the image of $x_i$ in $R$.
We may assume that $I=(\overline{x}_1,\dots,\overline{x}_r)$ for some $1\leq r \leq n$. If $r=1$ then $I=(\overline{x}_1)$ and we have $\H_1(I)=(0:_R\overline{x}_1)\K_1(I)=\fm^{p-1}\K_1(I)$.

Assume $r>1$ and set $J=(\overline{x}_1,\dots,\overline{x}_{r-1})$. There exists an exact sequence $$\dots \rightarrow\H_{i}(J)\overset{\overline{x}_r}\longrightarrow \H_i(J)\longrightarrow \H_i(I)\longrightarrow \H_{i-1}(J)\overset{\overline{x}_{r}}\longrightarrow \dots.$$
By induction hypothesis $f_i: \fm^{p-1}\K_i(J)\rightarrow \H_{i}(J)$ is surjective and so $\H_{i}(J)$ is a $k$-vector space for all $i\geq 1$. Thus the multiplication with $\overline{x}_r$ is zero map when $i\geq 1$. Therefore for $i\geq 2$ there is a commutative diagram
$$
\begin{CD}
0 @>>> \fm^{p-1}\K_i(J) @>>> \fm^{p-1}\K_i(I) @>>> \fm^{p-1}\K_{i-1}(J) @>>> 0\\
@. @Vf_i VV @Vg_iVV @Vf_{i-1} VV \\
0 @>>> \H_{i}(J) @>>> \H_i(I) @>>> \H_{i-1}(J) @>>> 0 .
\end{CD}
$$
By induction $f_{i-1}$ and $f_i$ are surjective and therefore $g_i$ is surjective for $i\geq 2$. When $i=1$ the last diagram turns to
$$
\begin{CD}
 0 @>>> \fm^{p-1}\K_1(J) @>>> \fm^{p-1}\K_1(I) @>>> \fm^{p-1} @>>> 0\\
@. @Vf_1 VV @Vg_1VV @Vf_{0} VV \\
 0 @>>> \H_{1}(J) @>>> \H_1(I) @>>> (0:_{R/J}\overline{x}_r) @>>> 0.
\end{CD}
$$

Note that $R/J\cong Q'/\fn'^p$ where $Q'=Q/(x_1,\dots,x_{r-1})$ is a regular local ring whose maximal ideal is $\fn'=(\overline{x}_r,\dots,\overline{x}_n)$. One observes that $f_0$ is a natural map of $k$-vector spaces such that $f_0(\overline{x}_r^{\alpha_r}\cdots\overline{x}_n^{\alpha_n})=\overline{x}_r^{\alpha_r}\cdots \overline{x}_n^{\alpha_n}$ where $\alpha_r+\dots+\alpha_n=p-1$. Therefore $f_0$ is surjective and since $f_1$ is surjective by induction, $g_1$ is surjective as well.
\end{proof} 

\begin{cor} Let $(Q,\fn,k)$ be a regular local ring. Let $R=Q/\fn^p$ where $p\geq 2$, and let $\fm$ be the maximal ideal of $R$. Then for any ideal $I$ of $R$ such that $I\cap \fm^2=\fm I$ one has $R\rightarrow R/I$ is a large homomorphism.
\end{cor}
\begin{proof} {It is well-known that $R$ is a Golod ring; see \cite[Theorem 4.2.6]{GL}.} Set $S=R/I$ and let $\overline{\fm}$ be the maximal ideal of $S$. {By Lemma \ref{power},} $ \fm^{p-1}\K_i(R)\rightarrow \H_i(R)$ and $\overline{\fm}^{p-1}\K_i(S)\rightarrow \H_i(S)$ are surjective for all $i\geq 1$. Since $\fm^{p-1}\K_i(R)\rightarrow \overline{\fm}^{p-1}\K_i(S)$ is also surjective, the commutative diagram
$$
\begin{CD}
 \fm^{p-1}\K_i(R) @>>> \H_i(R) \\
 @ VVV @ VVV \\
 \overline{\fm}^{p-1}\K_i(S) @>>> \H_i(S)
\end{CD}
$$
shows that the induced map $\H_i(R)\rightarrow \H_i(S)$ is surjective for all $i\geq 1$. Now the result follows from Proposition \ref{Massey} .
\end{proof}

\begin{ac}
The authors thank Rasoul Ahangari Maleki, David Jorgensen, and Liana \c{S}ega for their valuable suggestions and comments on the manuscript. The authors also thank the referee for providing very useful comments and suggestions that improved this article.
\end{ac}


\end{document}